\documentclass[dvipdfmx]{amsart}
\usepackage{amsfonts}
\usepackage{amssymb}
\usepackage{amsmath}
\usepackage{amsthm}
\usepackage{amscd}
\usepackage{mathrsfs}
\usepackage{graphicx, color}
\usepackage{url}

\allowdisplaybreaks[0]

\theoremstyle{theorem}
\newtheorem{theorem}{Theorem}[section]
\newtheorem{lemma}[theorem]{Lemma}
\newtheorem{corollary}[theorem]{Corollary}

\theoremstyle{definition}

\newtheorem{remark}[theorem]{Remark}

\newcommand{\vect}[1]{
 {\boldsymbol #1}
}

\begin{document}

\title[On region freeze crossing change]{A subspecies of region crossing change, \\ region freeze crossing change}

\author{Ayumu Inoue}
\address{Department of Mathematics Education, Aichi University of Education, Kariya, Aichi 448-8542, Japan}
\email{ainoue@auecc.aichi-edu.ac.jp}

\author{Ryo Shimizu}
\address{Okazakigakuen High School, Okazaki, Aichi 444-0071, Japan}
\email{s2120304@auecc.aichi-edu.ac.jp}


\begin{abstract}
We introduce a local move on a link diagram named a region freeze crossing change which is close to a region crossing change, but not the same.
We study similarity and difference between region crossing change and region freeze crossing change.
\end{abstract}

\maketitle

\section{Introduction}
\label{sec:introduction}

A \emph{region crossing change} at a region $R$ of a link diagram $D$ is a local move on $D$ which changes the crossings of $D$ touching $R$ (see Figure \ref{fig:RCC}).
Ayaka Shimizu \cite{Shimizu2014} showed that, on any knot diagram, a crossing change at a crossing is always realized by a sequence of region crossing changes.
Here, a sequence of local moves is said to \emph{realize a crossing change} at a crossing if the sequence and the crossing change bring the same effect.
Region crossing change is an unknotting operation for knots, because any link is untied by a sequence of crossing changes.
\begin{figure}[htbp]
 \begin{center}
  \includegraphics[scale=0.18]{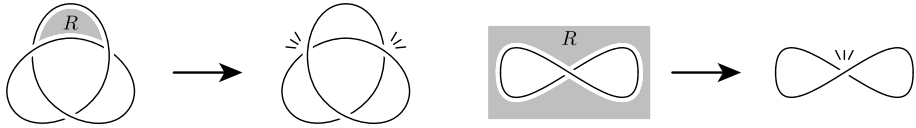}
 \end{center}
 \caption{Region crossing changes at $R$}
 \label{fig:RCC}
\end{figure}

Varieties of region crossing change have been proposed and studied by several authors \cite{AS2012, HSS2015}.
In this paper, we introduce another subspecies of region crossing change named region freeze crossing change.
A \emph{region freeze crossing change} at a region $R$ of a link diagram $D$ is a local move on $D$ which changes all crossings of $D$ other than the crossings touching $R$ (see Figure \ref{fig:RFCC}).
We will see that, just as region crossing change, region freeze crossing change is an unknotting operation for knots (Corollary \ref{cor:rfcc_is_unknotting_operation}).
On the other hand, in contrast with region crossing change, there is a knot diagram such that a crossing change at its certain crossing is not realized by any sequence of region freeze crossing changes.
We will give a necessary and sufficient condition for a knot diagram so that a crossing change at its crossing is always realized by a sequence of region freeze crossing changes (Theorems \ref{thm:main1} and \ref{thm:main2}).
\begin{figure}[htbp]
 \begin{center}
  \includegraphics[scale=0.18]{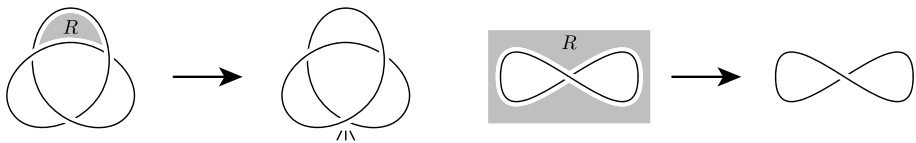}
 \end{center}
 \caption{Region freeze crossing changes at $R$}
 \label{fig:RFCC}
\end{figure}

Throughout this paper, a knot means a link with one component as usual.
For any subset $X, X^{\prime}$ of a set, $X \oplus X^{\prime}$ denotes the subset $(X \cup X^{\prime}) \setminus (X \cap X^{\prime})$.

\section{Facts about region crossing change}
\label{sec:facts_about_region_crossing_change}

In this section, we review several facts about region crossing change.
Obviously, region crossing changes first at a region and then at the region again do not change a link diagram in consequence.
Further effects of region crossing changes first at a region $R$ and then at a region $R^{\prime}$ and first at $R^{\prime}$ and then at $R$ are the same.
Thus we may reword region crossing changes at a region $R$, then at a region $R^{\prime}$, $\dots$, then at a region $R^{\prime \prime}$ as region crossing changes about $\{ R, R^{\prime}, \dots, R^{\prime \prime} \}$ shortly, if the regions $R, R^{\prime}, \dots, R^{\prime \prime}$ are mutually distinct.
Region crossing changes about $\emptyset$ shall mean that we do not change a link diagram at all.

Ayaka Shimizu showed the following key theorem:

\begin{theorem}[\cite{Shimizu2014}]
\label{thm:rcc_realizes_crossing_change}
On a knot diagram, a crossing change at a crossing is always realized by a sequence of region crossing changes.
\end{theorem}

\begin{proof}
We first see the case that a knot diagram is reduced.
Let $D$ be a reduced knot diagram and $c$ a crossing of $D$.
We choose an orientation of $D$ and splice $D$ at $c$.
As the result, we have a link diagram with two components.
Let $D_{1}$ and $D_{2}$ be the knot diagrams obtained from the link diagram by forgetting the other component.
We color the regions of $D_{1}$ with black and white like as a checkerboard.
Considering each region of $D_{1}$ as the corresponding regions of $D$, we apply region crossing changes on $D$ about all black regions.
Then the crossing change at $c$ is realized (see Figure \ref{fig:algorithm}).
Indeed, each crossing of $D$ persisting on $D_{1}$ touches just two black regions.
Since each crossing of $D$ persisting on $D_{2}$ lies in a region of $D_{1}$, it touches just zero or four black regions.
Further each of the other crossings of $D$ except $c$ touches just two black regions.
Thus the sequence of region crossing changes does not change those crossings in consequence.
On the other hand, the sequence changes the crossing $c$, because $c$ touches just one or three black regions.
\begin{figure}[htbp]
 \begin{center}
  \includegraphics[scale=0.18]{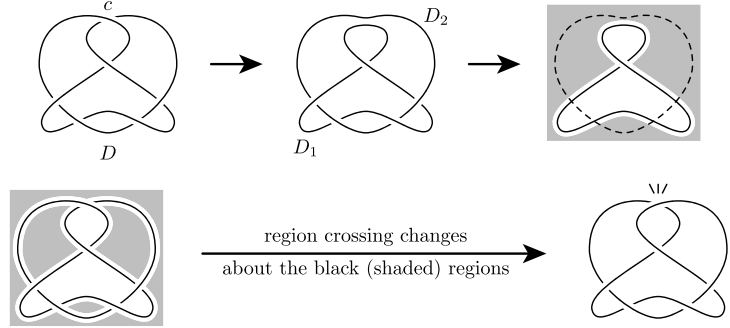}
 \end{center}
 \caption{Algorithm for a reduced knot diagram}
 \label{fig:algorithm}
\end{figure}

We next focus on the spacial case that a knot diagram $D$ has just one reducible crossing $c$; further we would like to realize a crossing change at $c$.
In this case, the above algorithm also works well if we color the regions of $D_{1}$ so that the color of the region including $D_{2}$ is white.
Then $c$ touches just one black region, and each of the other crossings does an even number of black regions.

We finally assume that the claim is true for any knot diagram with $k$ reducible crossings ($k \geq 0$).
Let $D$ be a knot diagram with $k+1$ reducible crossings and $c$ a crossing of $D$.
As long as $k \neq 0$ or $c$ is not a reducible crossing, we may choose a reducible crossing $c^{\prime}$ of $D$ differ from $c$ satisfying the following two conditions:
\begin{itemize}
\item
Let $D_{1}$ and $D_{2}$ be the knot diagrams obtained from $D$ by splicing it at $c^{\prime}$.
Then $D_{2}$ is irreducible (thus $D_{1}$ has $k$ reducible crossings).
\item
The crossing $c$ persists on $D_{1}$.
\end{itemize}
We note that $D_{2}$ lies in some region $R_{1}$ of $D_{1}$.
By the assumption, there is a set $\mathcal{R}_{1}$ consisting of regions of $D_{1}$ such that region crossing changes about $\mathcal{R}_{1}$ realize the crossing change at $c$ on $D_{1}$.
Let $\mathcal{R}$ be the set consisting of regions of $D$, each of which coincides with or lies in a region in $\mathcal{R}_{1}$.
If $c^{\prime}$ touches an even number of regions in $\mathcal{R}$, then region crossing changes about $\mathcal{R}$ realize the crossing change at $c$ on $D$.
Indeed, each crossing of $D$ persisting on $D_{2}$ touches just zero or four regions in $\mathcal{R}$.
Otherwise, if $c^{\prime}$ touches an odd number of regions in $\mathcal{R}$, we color the regions of $D_{2}$ with black and white like as a checkerboard so that the color of the region including $D_{1}$ is white.
Considering each region of $D_{2}$ as the corresponding regions of $D$, we subtract or add all black regions from or to $\mathcal{R}$ if $R_{1}$ is a member of $\mathcal{R}_{1}$ or not respectively.
Then region crossing changes about $\mathcal{R}$ realize the crossing change at $c$ on $D$.
Indeed, each crossing of $D$ persisting on $D_{2}$ touches just two regions in $\mathcal{R}$.
By the assumption for the coloring, the subtraction or addition does not change each number of regions in $\mathcal{R}$ which a crossing of $D$ persisting on $D_{1}$ touches.
Further $c^{\prime}$ touches just two regions in $\mathcal{R}$.
\end{proof}

Since any link is untied by a sequence of crossing changes, we immediately have the following corollary:

\begin{corollary}[\cite{Shimizu2014}]
\label{cor:rcc_is_unknotting_operation}
Region crossing change is an unknotting operation for knots.
\end{corollary}

Although region crossing change is not an unknotting operation for links with two or more components, Cheng ZhiYun and Gao HongZhu gave the following criteria (we omit the proof in this paper):

\begin{theorem}[\cite{ZhiYun2013, ZH2012}]
\label{thm:link_case}
A sequence of region crossing changes unties a link $L$ with two or more components if and only if the total linking number of $L$ is even.
\end{theorem}

A set $\mathcal{R}^{\ast}$ consisting of regions of a link diagram is said to be \emph{ineffective} if region crossing changes about $\mathcal{R}^{\ast}$ do not change the diagram in consequence.

\begin{lemma}
\label{lem:existence_of_four_ineffective_region_sets}
We have at least four ineffective region sets for any knot diagram.
\end{lemma}

\begin{proof}
Obviously, $\mathcal{R}^{\ast}_{1} = \emptyset$ is ineffective for any link diagram.

We first let $D$ be a reduced knot diagram.
We color the regions of $D$ by black and white like as a checkerboard.
Suppose $\mathcal{R}^{\ast}_{2}$ (resp.\ $\mathcal{R}^{\ast}_{3}$) is the set consisting of all black (resp.\ white) regions.
Then $\mathcal{R}^{\ast}_{2}$, $\mathcal{R}^{\ast}_{3}$ and $\mathcal{R}^{\ast}_{4} = \mathcal{R}^{\ast}_{2} \oplus \mathcal{R}^{\ast}_{3}$ are mutually distinct, non-empty and ineffective for $D$.

We next let $D$ be a reducible knot diagram with $k$ reducible crossings.
Splicing $D$ at all reducible crossings, we obtain reduced knot diagrams $D_{i}$ ($i = 1, 2, \dots, k$).
Suppose $\mathcal{R}^{\ast}_{2}$, $\mathcal{R}^{\ast}_{3}$ and $\mathcal{R}^{\ast}_{4}$ are the above ineffective region sets for some $D_{j}$.
Then we may extend $\mathcal{R}^{\ast}_{2}$, $\mathcal{R}^{\ast}_{3}$ and $\mathcal{R}^{\ast}_{4}$ to mutually distinct, non-empty and ineffective region sets for $D$ along the same lines of the proof of Theorem \ref{thm:rcc_realizes_crossing_change}.
\end{proof}

An effect of region crossing changes can be estimated algebraically as follows.
Let $D$ be a link diagram, $c_{1}, c_{2}, \dots, c_{n}$ the crossings of $D$, and $R_{1}, R_{2}, \dots, R_{m}$ the regions of $D$.
The transpose of the \emph{incidence matrix}\footnote{The incidence matrix of $D$ was introduced by Cheng ZhiYun and Gao HongZhu \cite{ZH2012}. Although the transpose of $A$ coincides with the incidence matrix of $D$ after shuffling its rows in general, we simply call $A$ the transpose of the incidence matrix in this paper.} of $D$ is the $n \times m$ matrix $A$ over $\mathbb{Z} / 2 \mathbb{Z}$ whose $(i, j)$ entry is $1$ if $c_{i}$ touches $R_{j}$, $0$ otherwise.
For a set $\mathcal{R}$ consisting of regions of $D$, consider the $m$-dimensional vector $\vect{r}$ over $\mathbb{Z} / 2 \mathbb{Z}$ whose $i$-th entry is $1$ if $R_{i}$ in $\mathcal{R}$, $0$ otherwise.
Then region crossing changes about $\mathcal{R}$ changes $c_{i}$ in consequence if and only if the $i$-th entry of $A \vect{r}$ is 1.
We have the following theorem, which is a special case of the work of Megumi Hashizume:

\begin{theorem}[\cite{Hashizume2015}]
\label{thm:four_solutions}
If $D$ is a knot diagram, then we have just four solutions of the linear equations $A \vect{r} = \vect{c}$ over $\mathbb{Z} / 2 \mathbb{Z}$ for each $n$-dimensional vector $\vect{c}$ over $\mathbb{Z} / 2 \mathbb{Z}$.
\end{theorem}

\begin{proof}
We first note that $m = n + 2$ in this case.
Theorem \ref{thm:rcc_realizes_crossing_change} ensures that the rank of $A$ is $n$.
Further Lemma \ref{lem:existence_of_four_ineffective_region_sets} claims that we have at least four solutions of the linear equations $A \vect{r} = \vect{o}$, where $\vect{o}$ denotes the zero vector.
Thus, reordering the indices of the regions if necessary, $A$ is transformed into a reduced row echelon form $(I_{n} \ \vect{x} \ \vect{y})$ with some non-zero vectors $\vect{x}$ and $\vect{y}$ by a sequence of elementary row operations, where $I_{n}$ denotes the $n$-dimensional identity matrix.
\end{proof}

The claim in Theorem \ref{thm:four_solutions} is able to be rephrased as follows.
If $D$ is a knot diagram, for any set $\mathcal{C}$ of crossings of $D$, there are just four sets $\mathcal{R}_{1}$, $\mathcal{R}_{2}$, $\mathcal{R}_{3}$ and $\mathcal{R}_{4}$ consisting of regions of $D$ such that region crossing changes about each $\mathcal{R}_{i}$ change all crossings in $\mathcal{C}$ but do not the other crossings in consequence.
Further let $\mathcal{R}^{\ast}_{2}$, $\mathcal{R}^{\ast}_{3}$ and $\mathcal{R}^{\ast}_{4}$ be the non-empty ineffective region sets for $D$.
Then $\mathcal{R}_{i}$ coincide with $\mathcal{R}_{1} \oplus \mathcal{R}^{\ast}_{\sigma(i)}$ with some $\sigma \in \mathfrak{S}_{3}$, where $\mathfrak{S}_{3}$ denotes the symmetric group on $\{ 2, 3, 4 \}$.

\begin{remark}
Assume that $D$ is a knot diagram.
For each point $p$ on an arc of $D$ differ from crossings, let $R$ and $R^{\prime}$ be the regions of $D$ touching $p$.
Since $D$ has exactly four ineffective region sets $\mathcal{R}^{\ast}_{1}$, $\mathcal{R}^{\ast}_{2}$, $\mathcal{R}^{\ast}_{3}$ and $\mathcal{R}^{\ast}_{4}$, as was claimed in \cite{Hashizume2015}, just two of $\{ \mathcal{R}^{\ast}_{1}, \mathcal{R}^{\ast}_{2}, \mathcal{R}^{\ast}_{3}, \mathcal{R}^{\ast}_{4}\}$ have $R$ (resp.\ $R^{\prime}$) as a member.
Further only one of them has both $R$ and $R^{\prime}$ as members.
Recall that the non-empty ineffective region sets for $D$ come from a checkerboard coloring of an arbitrary reducible part of $D$.
Suppose $(I_{n} \ \vect{x} \ \vect{y})$ is the reduced row echelon form of $A$.
We may assume that $R_{n+1} = R$, $R_{n+2} = R^{\prime}$, and $\mathcal{R}^{\ast}_{2}$ (resp.\ $\mathcal{R}^{\ast}_{3}$) has $R$ (resp.\ $R^{\prime}$) as a member but does not $R^{\prime}$ (resp.\ $R$).
Then $\vect{x}$ (resp.\ $\vect{y}$) is the vector whose $i$-th entry is 1 if $R_{i}$ is a member of $\mathcal{R}^{\ast}_{2}$ (resp.\ $\mathcal{R}^{\ast}_{3}$), $0$ otherwise.
\end{remark}

\section{Study on region freeze crossing change}
\label{sec:study_on_region_freeze_crossing_change}

In this section, we study on region freeze crossing change.
Just as region crossing change, region freeze crossing changes first at a region and then at the region again do not change a link diagram in consequence.
Effects of region freeze crossing changes first at a region $R$ and then at a region $R^{\prime}$ and first at $R^{\prime}$ and then at $R$ are the same.
We thus reword region freeze crossing changes at a region $R$, then at a region $R^{\prime}$, $\dots$, then at a region $R^{\prime \prime}$ as region freeze crossing changes about $\{ R, R^{\prime}, \dots, R^{\prime \prime} \}$ if the regions $R, R^{\prime}, \dots, R^{\prime \prime}$ are mutually distinct.

We first see a relationship between region freeze crossing change and region crossing change:

\begin{lemma}
\label{lem:relationship}
Let $D$ be a link diagram and $\mathcal{R}$ a set consisting of regions of $D$.
If the cardinality of $\mathcal{R}$ is even, then effects of region freeze crossing changes about $\mathcal{R}$ and region crossing changes about $\mathcal{R}$ are the same.
Otherwise, the diagram obtained from $D$ by applying region freeze crossing changes about $\mathcal{R}$ is the mirror image of the diagram obtained from $D$ by applying region crossing changes about $\mathcal{R}$.
\end{lemma}

\begin{proof}
By definition, the diagram obtained from $D$ by applying a region freeze crossing change at a region $R$ of $D$ is the mirror image of the diagram obtained from $D$ by applying the region crossing change at $R$.
Since taking a mirror image twice does not change a link diagram in consequence, we have the claim.
\end{proof}

Lemma \ref{lem:relationship} immediately gives us the following corollary:

\begin{corollary}
\label{cor:rfcc_is_unknotting_operation}
Region freeze crossing change is an unknotting operation for knots.
\end{corollary}

\begin{proof}
In light of Corollary \ref{cor:rcc_is_unknotting_operation}, a knot diagram $D$ is transformed into a diagram $D^{\prime}$ of the unknot by region crossing changes about a certain set $\mathcal{R}$ consisting of regions of $D$.
Lemma \ref{lem:relationship} says that we obtain $D^{\prime}$ or its mirror image after applying region freeze crossing changes on $D$ about $\mathcal{R}$.
Both of those are of course diagrams of the unknot.
\end{proof}

Similarly, in light of Theorem \ref{thm:link_case}, we have the following corollary:

\begin{corollary}
A link $L$ with two or more components is untied by a sequence of region freeze crossing changes if and only if the total linking number of $L$ is even.
\end{corollary}

Lemma \ref{lem:relationship} further gives us the following theorem:

\begin{theorem}
\label{thm:main1}
Let $D$ be a knot diagram and $\mathcal{R}^{\ast}_{1}$, $\mathcal{R}^{\ast}_{2}$, $\mathcal{R}^{\ast}_{3}$ and $\mathcal{R}^{\ast}_{4}$ the ineffective region sets for $D$.
If the cardinality of $\mathcal{R}^{\ast}_{j}$ is odd for some $j$, then a crossing change at a crossing of $D$ is always realized by a sequence of region freeze crossing changes.
\end{theorem}

\begin{proof}
For each crossing $c$ of $D$, in light of Theorem \ref{thm:rcc_realizes_crossing_change}, there is a set $\mathcal{R}$ consisting of regions of $D$ such that region crossing changes about $\mathcal{R}$ realize the crossing change at $c$.
If the cardinality of $\mathcal{R}$ is even, then region freeze crossing changes about $\mathcal{R}$ realize the crossing change at $c$.
Otherwise, since the cardinality of $\mathcal{R} \oplus \mathcal{R}^{\ast}_{j}$ is even, region freeze crossing changes about $\mathcal{R} \oplus \mathcal{R}^{\ast}_{j}$ realize the crossing change at $c$.
Recall that region crossing changes about $\mathcal{R} \oplus \mathcal{R}^{\ast}_{j}$ also realize the crossing change at $c$.
\end{proof}

Of course, even though all cardinalities of the ineffective region sets for a knot diagram are even, we possibly have a set consisting of an even number of regions of the knot diagram about which region crossing changes realize a crossing change.

\begin{theorem}
\label{thm:main2}
Let $D$ be a knot diagram.
Assume that all cardinalities of the ineffective region sets for $D$ are even.
Then a crossing change at a crossing of $D$ is always realized by a sequence of region freeze crossing changes if and only if the number of crossings of $D$ satisfying the following condition is even:
\begin{itemize}
\item[($\star$)]
The crossing change at the crossing is realized by region crossing changes about a set consisting of an odd number of regions
{\upshape (}then all cardinalities of the four region sets, about each of which region crossing changes realize the crossing change at the crossing, are odd by the assumption{\upshape )}.
\end{itemize}
\end{theorem}

\begin{proof}
Obviously, a crossing change at a crossing of $D$ which does not satisfy ($\star$) is realized by some sequence of region freeze crossing changes.
We thus see that the crossing change at a crossing $c$ of $D$ satisfying ($\star$) is realized by a sequence of region freeze crossing changes, if the number of crossings satisfying ($\star$) is even.
Let $\mathcal{R}_{1}, \mathcal{R}_{2}, \dots, \mathcal{R}_{n-1}$ be sets consisting of regions of $D$ about which region crossing changes realize crossing changes at the crossings of $D$ other than $c$ respectively.
Then region crossing changes about $\mathcal{R} = \mathcal{R}_{1} \oplus \mathcal{R}_{2} \oplus \dots \oplus \mathcal{R}_{n-1}$ change all crossings of $D$ other than $c$ in consequence.
Since the cardinality of $\mathcal{R}$ is odd by the assumption, region freeze crossing changes about $\mathcal{R}$ change only the crossing $c$ in consequence.

On the other hand, assume that there is a set $\mathcal{R}$ consisting of regions of $D$ about which region freeze crossing changes realize a crossing change at a crossing $c$ of $D$ satisfying ($\star$).
Then region crossing changes about $\mathcal{R}$ change all crossings of $D$ other than $c$ in consequence, because the cardinality of $\mathcal{R}$ should not be even.
Thus the number of crossings of $D$ other than $c$ satisfying ($\star$) must be odd.
\end{proof}

We wrap up our study with an example of a knot diagram having a crossing at which any sequence of region freeze crossing changes does not realize the crossing change.
Let $D$ be the knot diagram illustrated in the left-hand side of Figure \ref{fig:example1}.
Figure \ref{fig:example1} also depicts the ineffective region sets $\mathcal{R}^{\ast}_{1}$, $\mathcal{R}^{\ast}_{2}$, $\mathcal{R}^{\ast}_{3}$ and $\mathcal{R}^{\ast}_{4}$ for $D$ as shaded regions respectively.
We note that each cardinality of $\mathcal{R}^{\ast}_{i}$ is even.
It is routine to check that region crossing changes about shaded regions in Figure \ref{fig:example2} realize crossing changes at crossings $c_{i}$ of $D$ respectively.
We note that the crossings $c_{2}$, $c_{4}$ and $c_{8}$ satisfy the condition ($\star$); and the other crossings of $D$ do not.
Therefore, in light of Theorem \ref{thm:main2}, the crossing change at $c_{2}$, $c_{4}$ or $c_{8}$ is not realized by any sequence of region freeze crossing changes on $D$.
\begin{figure}[htbp]
 \begin{center}
  \includegraphics[scale=0.18]{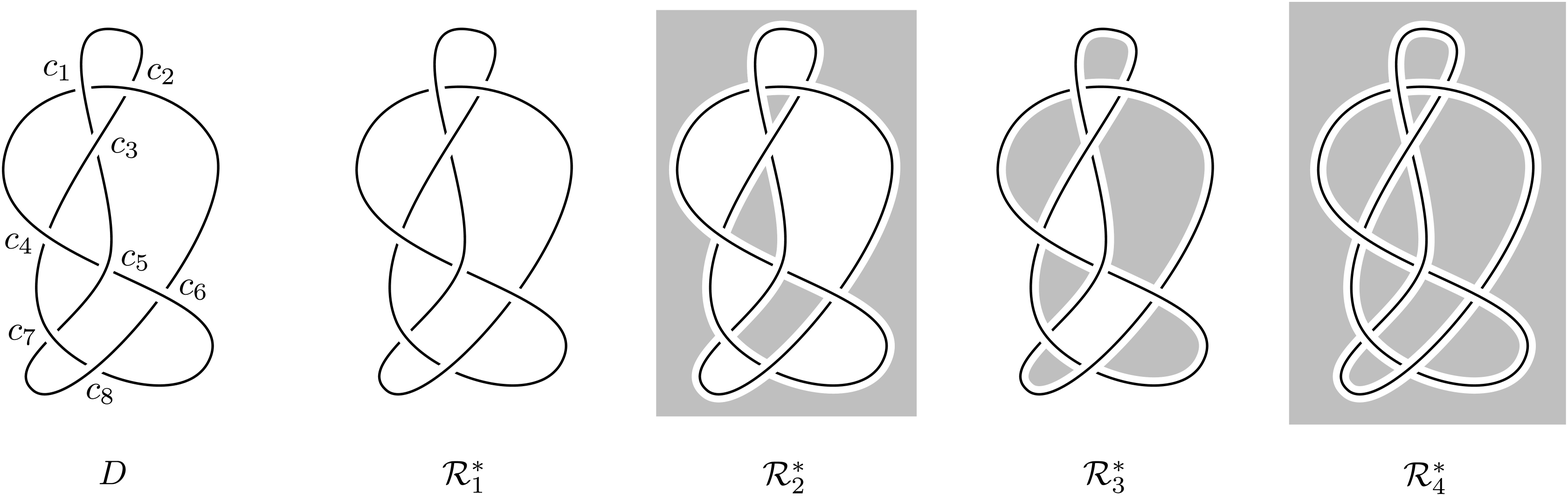}
 \end{center}
 \caption{A diagram of the $8_{13}$ knot and its four ineffective region sets}
 \label{fig:example1}
\end{figure}
\begin{figure}[htbp]
 \begin{center}
  \includegraphics[scale=0.18]{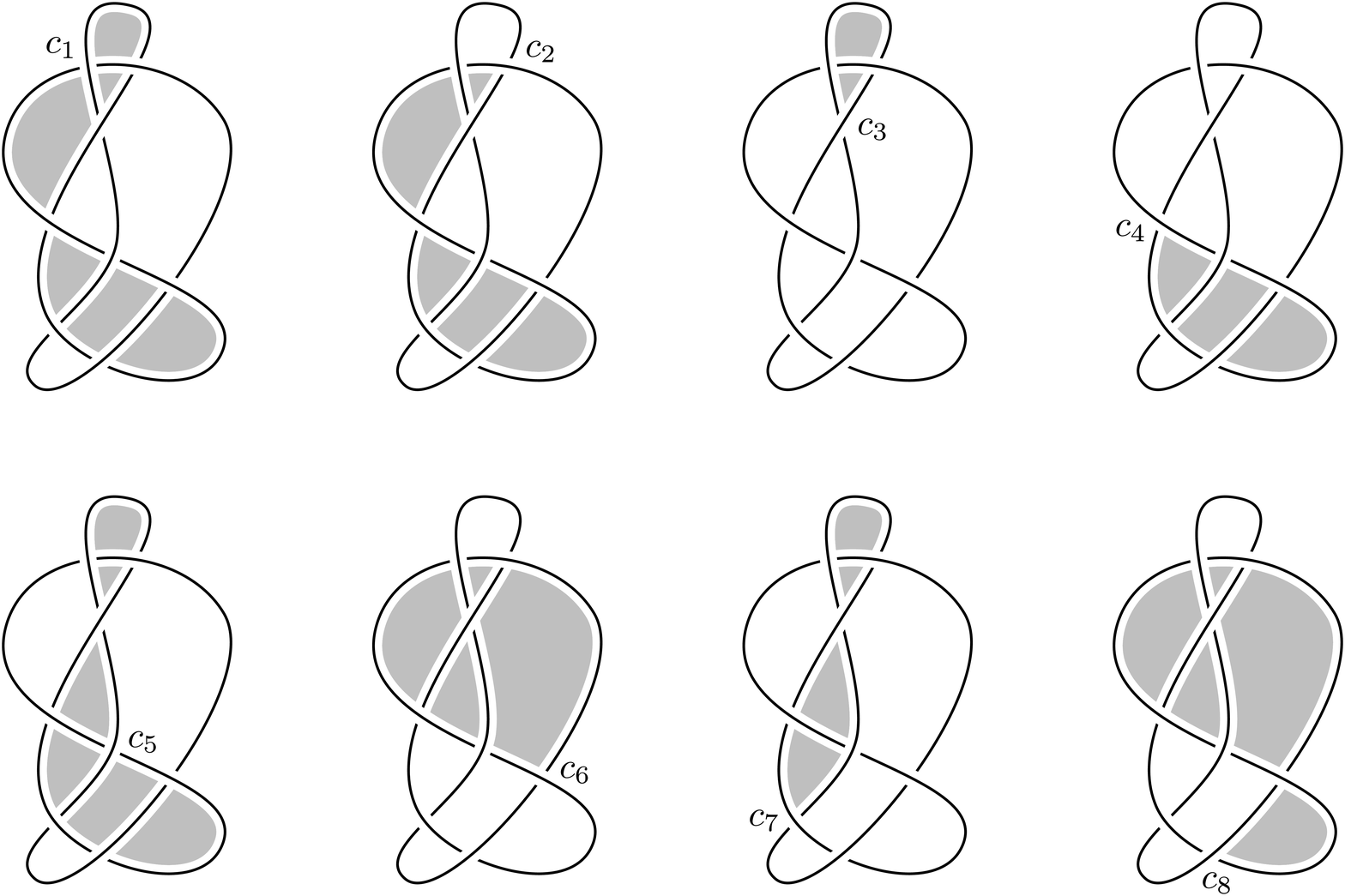}
 \end{center}
 \caption{Region crossing changes about shaded regions realize crossing changes at $c_{i}$ respectively}
 \label{fig:example2}
\end{figure}

\section*{Acknowledgments}
The notion of region freeze crossing change was originally proposed by Masato Kida.
The authors would like to express their sincere gratitude to him.
The first author is partially supported by JSPS KAKENHI Grant Number 16K17591.

\bibliographystyle{amsplain}

\end{document}